\documentclass{scrartcl}
\usepackage{amsmath,amssymb,amsthm}
\usepackage{mathtools}

\usepackage[final]{graphicx}
\usepackage{caption}
\usepackage{subcaption}
\usepackage{color}

\theoremstyle{plain}
\newtheorem{theorem}{Theorem}[section]
\newtheorem{lemma}[theorem]{Lemma}
\newtheorem{proposition}[theorem]{Proposition}
\newtheorem{corollary}[theorem]{Corollary}

\theoremstyle{definition}

\newtheorem{example}{Example}[section]

\theoremstyle{remark}
\newtheorem*{remark}{Remark}

\newcommand{\Z}{\mathbb{Z}}

\newcommand{\R}{\mathbb{R}}
\newcommand{\C}{\mathbb{C}}
\newcommand{\mc}{\mathcal}
\newcommand{\la}{\langle}
\newcommand{\ra}{\rangle}

\newcommand*\diff{\mathop{}\!\mathrm{d}}

\begin{document}
\begin{center}
{\Huge
Sparse Control of Quantum Systems}
\normalsize
\\[4mm]
\textsc{Gero Friesecke$^*$, Felix Henneke$^{*}$, and Karl Kunisch$^{**}$} \\[4mm]
$^*$Faculty of Mathematics, TU Munich, Germany \\ \texttt{gf@ma.tum.de, henneke@ma.tum.de} \\
$^{**}$Institute for Mathematics and Scientific Computing, University of Graz, Austria \\
\texttt{karl.kunisch@uni-graz.at} \\
\end{center}

\vspace{0.4in}

\noindent {\normalsize \textbf{ Abstract.} A new class of cost functionals for optimal control of quantum systems which produces controls which are sparse in frequency and smooth in time is proposed.
This is achieved by penalizing a suitable time-frequency representation of the control field, rather than the control field itself, and by employing norms which are of $L^1$ or measure form with respect to frequency but smooth with respect to time.

We prove existence of optimal controls for the resulting nonsmooth optimization problem, derive necessary optimality conditions, and rigorously establish the frequency-sparsity of the optimizers. More precisely, we show that the time-frequency representation of the control field, which a priori admits a continuum of frequencies, is supported on only \textit{ finitely many} frequencies.
These results cover important systems of physical interest, including (infinite-dimensional) Schr\"odinger dynamics on multiple potential energy surfaces as arising in laser control of chemical reactions.
Numerical simulations confirm that the optimal controls, unlike those obtained with the usual $L^2$ costs, concentrate on just a few frequencies, even in the infinite-dimensional case of laser-controlled chemical reactions.
\vspace{0.4in}

\section{Introduction}
This paper is motivated by application problems of current interest in quantum control~\cite{Getal:2015:EPJD}, which range from steering chemical reactions~\cite{BKZB:2008:ACP} over creating excited or ionized states~\cite{HRG:2013:PRA} to faithfully storing and manipulating bits of quantum information \cite{SKSCKMDHMCNJ:2014:NJP}.

We propose a new class of cost functionals for the optimal control of quantum systems which result in controls with a sparse time-frequency structure.
This is achieved via two key ideas.

First, we do not penalize the time profile of the field amplitude but a suitable \emph{time-frequency representation} of it.
While such representations are a familiar tool to interpreting or analyzing a given field in control and signal analysis, they here acquire center stage already in the design of the controls.

Second, we build upon recent advances in the optimal control theory of elliptic and parabolic systems related to the basic idea \cite{VM:2006:OCAM,S:2009:COAP,CK:2011:COCV,CHW:2012:SJO,HSW:2012:SJCO} of \emph{sparsity-enhancing $L^1$ or measure-norm costs}.
More specifically, we build upon the idea of
function-valued measures to achieve directional sparsity in parabolic control~\cite{KPV:2014:SJCO}.
The novelty as compared to the latter advances is two-fold: first, in quantum control, unlike in parabolic control, the target for sparsity should not be the field amplitude, but its frequency structure,
and second, one is dealing with a bilinear instead of a linear control problem.

These ideas result in constrained non-smooth optimization problems of the form
\begin{equation} \label{intro:1}
  \mbox{Minimize } \frac12 \langle\psi(T), {\mc O}\psi(T)\rangle + \alpha ||u||_{{\mc M}} \mbox{ over controls } u\colon \Omega\times[0,T]\to \C
\end{equation}
subject to
\begin{equation} \label{intro:2}
  i\partial_t\psi = (H_0 + (Bu)(t)H_1)\psi, \;\;
  \psi(0)=\psi_0.
\end{equation}
Here $H_0$ is the Hamiltonian of the quantum system, $B$ is a `synthesis operator' which assembles the control field from a time-frequency representation $u(\omega, t)$, and $\| \, \cdot \, \|_{{\mc M}}$ is an $L^1$ or measure norm with respect to frequency but a smoothness-promoting norm with respect to time.
A prototypical choice is the time-frequency synthesis operator
\begin{equation} \label{intro:B}
  (Bu)(t) = \int_{\Omega} u(\omega,t) \, e^{i\omega t} \diff\omega,
\end{equation}
where $\Omega\subset\R$ is a region of admissible frequencies, and
\begin{equation} \label{intro:M}
  \lVert u\rVert_{{\mc M}} = \int_{\Omega} \lVert u(\omega,\cdot)\rVert_{H^1(0,T)} \diff\omega.
\end{equation}
Note that the control, a priori, can use the whole available continuum of frequencies, with each frequency possessing its own time profile.
A main result of this paper (Theorem~\ref{thm:finite_supp}) is that the optimizers utilize only \textit{ finitely many} frequencies, even when the quantum dynamics is a full infinite-dimensional Schr\"odinger dynamics on multiple potential energy surfaces (Ex.~\ref{ex:mol2} and Section~\ref{sec:numerics_mol2}).
Equations \eqref{intro:B}--\eqref{intro:M} replace the standard approach in quantum optimal control initiated by \cite{PDR:1988:PR} to penalize just the $L^2$ or $H^1$ norm of the field amplitude (see \cite{IK:2007:SJCO,vWB:2008:IP,HMMS:2013:SJCO} for mathematical results).

Numerical simulations presented in Section~\ref{sec:numerics} below illustrate that
the optimizers concentrate on just a few frequencies, even when the quantum dynamics is a full infinite-dimensional Schr\"odinger dynamics.
Thus our controls share an important feature of laser pulses designed by experimentalists.

The plan of this paper is as follows.
In the next section we introduce three important physical examples of
controlled quantum dynamical systems
which motivated this work and to which our sparsity results apply. These examples also serve to recall basic features of the coupling operators such as `forbidden transitions' and the oscillatory nature of quantum controls.
In Section~\ref{sec:control} we introduce our measure-norm sparsity-enhancing costs within a general functional-analytic framework, and give several examples.
In particular, the choices \eqref{intro:B}--\eqref{intro:M} lead to frequency-sparsity with global time profiles, and appropriate modifications lead to frequency-sparsity with local time profiles.
Section~\ref{sec:theory} is devoted to the mathematical analysis of the non-smooth optimal control problem \eqref{intro:1}--\eqref{intro:2}.
We recall the relevant well known results on the existence of dynamics, establish existence of optimal controls, derive necessary optimality conditions, and prove that optimal controls are supported on only finitely many frequencies.
Finally, in Section~\ref{sec:numerics} we numerically calculate optimal controls and compare them to those obtained from the usual $L^2$ penalization of the field amplitude.
Specifically, we present a 3-level example which arises in atomic excitation problems, and an example of Schr\"odinger dynamics on two potential energy surfaces as arising in laser-controlled chemical reaction dynamics.
\section{Quantum dynamics with controls: physical examples}\label{sec:models}
The evolution equations in quantum control problems typically have the structure
\begin{equation} \label{eq:evolution}
  i \partial_t\psi(t) = \left( H_0 + \sum_{\ell=1}^L v_\ell(t)H_\ell\right)\psi(t),
\end{equation}
where the state $\psi(t)$ belongs to some Hilbert space $\mc H$, $H_0$ and the $H_\ell$ are (bounded or unbounded) self-adjoint operators on $\mc H$, and the $v_\ell(t)$ are real-valued scalar amplitudes of components of applied electric or magnetic fields.
The operator $H_0$ is the Hamiltonian of the system in the absence of fields, and the operators $H_\ell$ describe the system-field coupling.

A general mathematical reference for finite-dimensional problems of form~\eqref{eq:evolution} is~\cite{D:2008:CH}.
In the mathematical control theory literature, general aspects of infinite-dimensional problems of the above form such as existence of optimal controls with $L^2$ or $H^1$ penalization of the field have been previously treated~\cite{IK:2007:SJCO,vWBV:2009:SJSC}.
Here our goal is to mathematically understand the highly oscillatory nature of optimal controls, unfamiliar from elliptic and parabolic control problems but arising in the physics of \emph{Bohr frequencies}, and to develop novel penalty terms to simplify this oscillatory structure.

We note that equation~\eqref{eq:evolution} already contains two important approximations which are valid in many situations of interest.
First, quantum fluctuations of the field amplitudes can be neglected, that is to say we are dealing with classical fields and do not need to move to the much more complicated framework of quantum field theory.
Second, the spatial wavelength of the applied fields is much larger than the localization length of the state $\psi(t)$, so that it is sufficient to assume that the field strengths depend on time only.
This is often called `dipole approximation', see e.g.~\cite{S:1990:W}.

We now give three examples for~\eqref{eq:evolution} of physical interest.
The first one has been included to recall the quantum mechanical meaning of oscillatory controls which forms the starting point for the time-frequency approach developed here; the other two will be used in the simulations in this paper.
\begin{example}
(Spin of a spin 1/2 particle in a magnetic field)
This is the simplest control system of physical interest.
It arises as a basic example in NMR, and more recently as a model of a single qubit in quantum information theory (see~\cite{SKSCKMDHMCNJ:2014:NJP} for a recent careful experimental realization of this system and~\cite{BCGGJ:2002:JMP} for rigorous mathematical results).
It already exhibits surprisingly many features of complex systems.
The spin at time $t$ is a unit vector in the Hilbert space $\mc H=\C^2$.
The general evolution equation of a spin in a time-dependent magnetic field $B \colon \R \to \R^3$ is
\[
  i\partial_t \psi(t) = - \gamma B(t)\cdot S \, \psi(t) = - \gamma \sum_{\ell=1}^3 B_\alpha(t) S_\alpha \, \psi(t)
\]
where the component operators $S_\alpha$ of the spin operator $S$ are given by $\hbar/2$ times the Pauli matrices, i.e.\ in atomic units ($\hbar=1$)
\[
  S_1 = \frac12 \begin{pmatrix} 0 & 1 \\ 1 & 0\end{pmatrix}, \quad
  S_2 = \frac12 \begin{pmatrix} 0 & -i \\ i & 0\end{pmatrix}, \quad
  S_3 = \frac12 \begin{pmatrix} 1 & 0 \\ 0 & -1\end{pmatrix}.
\]
The factor $\gamma$ depends on the type of particle (electron, proton, neutron, nucleus) and can be positive or negative.
A typical control problem in NMR consists of taking $B_3$ time-independent and comparatively large, and $B_1$ and $B_2$ as time-dependent controls which are small.
This is a system of form~\eqref{eq:evolution}, with $H_0 = -\gamma B_3 S_3$.
Denoting the two eigenvalues of $H_0$ by $E_1$, $E_2$, this system can be written in the compact form
\[
  i\partial_t\psi = \begin{pmatrix} E_1 & 0 \\ 0 & E_2\end{pmatrix} \psi +
                    \begin{pmatrix} 0 & v^*(t) \\ v(t) & 0 \end{pmatrix}\psi
\]
with complex-valued control $v(t) = -2\gamma(B_1(t) + i B_2(t))$, and $v^*$ its complex conjugate.
The basic case of a time-harmonic control field $v(t)=Ae^{i\omega t}$ is exactly soluble~\cite{R:1937:PR}.
This allows one to understand mathematically the emergence of \emph{oscillatory controls} and \emph{Bohr frequencies}.
The Bohr frequency of a transition between two quantum states is the eigenvalue difference.
The time-harmonic control with this frequency, when applied over a time window of suitable length, achieves a 100 $\%$ transfer; and it is the \emph{only} time-harmonic control which achieves a 100 $\%$ transfer~\cite{R:1937:PR}.
For more general quantum systems relations of this kind typically only hold asymptotically~\cite{C:2012:A}.
\end{example}
\begin{example}\label{ex:atom}
(Electronic states of atoms in laser fields)
A standard reference in the physics literature is~\cite{S:1990:W}.
Consider an atom with N electrons of charge $e = -1$ and a nucleus of charge $Z=N$ clamped at the origin.
The electronic state of the atom is described by a function belonging to the Hilbert space $\mc H =\{\psi\in L^2({(\R^3\times\Z_2)}^N) \, | \, \psi \mbox{ antisymmetric}\}$.
That is to say electronic states are functions $\psi=\psi(x_1,s_1,\dotsc,x_N,s_N)$ which depend on the position coordinates $x_i\in\R^3$ and the spin coordinates $s_i\in\Z_2$ of all the electrons.
An applied electric field can be described by a function $E \colon \R\to\R^3$, with $E(t)$ denoting the electric field vector at time $t$.
(Here quantum fluctuations of the field as well as its spatial dependence are neglected, as discussed above.)
The overall evolution equation is
\begin{equation}\label{eq:atomevol}
  i \partial_t \psi(t) = \left(-\frac12\Delta + V(x_1,\dotsc,x_N) - E(t)\cdot D(x_1,\dotsc,x_N) \right)\psi(t),
\end{equation}
where $\Delta$ is the Laplacian on $\R^{3N}$, the many-body Coulomb potential $V$ is given by $V(x_1,\dotsc,x_N) = - \sum_{i=1}^N Z / |x_i| + \sum_{1\le i<j\le N} 1 / |x_i-x_j|$, and $D$ is the dipole operator $D(x_1,\dotsc,x_N) = \sum_{i=1}^N  e \, x_i$.
This has the form (\ref{eq:evolution}), as is immediate by denoting the components of $E(t)$ and $x_i$ with respect to some orthonormal basis of $\R^3$ by $E_\ell(t)$ and $x_{i\ell}$ ($\ell=1,2,3$), and letting $H_\ell =\sum_{i=1}^N x_{i\ell}$.
Because the high-dimensional Schr\"odinger equation~\eqref{eq:atomevol} cannot be simulated in practice, the infinite-dimensional state equation~\eqref{eq:atomevol} is often replaced by projecting onto finitely many eigenstates $\psi_1,\dotsc,\psi_d$ of $H_0 = -\frac{1}{2}\Delta + V$, i.e.\ $\psi(t) \approx a_1(t)\psi_1+\dotsb+a_d(t)\psi_d$ and neglecting the coupling with the rest of the system.
Assuming for simplicity that the field is unidirectional, i.e.\ $E(t)=v(t)E_0$ for some unit vector $E_0$ and a scalar amplitude $v$, and denoting the eigenvalues of $H_0$ corresponding to the states $\psi_1,\dotsc,\psi_d$ by $E_1,\dotsc,E_d$, this yields
\begin{equation}\label{eq:evolred2}
  i\partial_t \begin{pmatrix} a_1 \\ a_2 \\ \vdots \\ a_d \end{pmatrix} =
  \left( \begin{pmatrix} E_1 & 0 & \cdots & 0 \\
                   0  & E_2 & \cdots  & 0 \\
                   \vdots &  & \vdots \\
                   0 & \cdots & & E_d \end{pmatrix} + v(t)
  \begin{pmatrix} 0 & \mu_{12} & \cdots & \mu_{1d} \\
                   \mu_{12}^*  &  0 & \cdots  & \mu_{2d} \\
                   \vdots & \vdots & \vdots \\
                   \mu_{1d}^* & \mu_{2d}^* & \cdots & 0 \end{pmatrix} \right)
  \begin{pmatrix} a_1 \\ a_2 \\ \vdots \\ a_d \end{pmatrix},
\end{equation}
with the coupling matrix elements
\begin{equation}\label{eq:dip}
  \mu_{mn} = \langle \psi_m, -E_0\cdot D \, \psi_n\rangle.
\end{equation}
We remark that the off-diagonal element $\mu_{mn}$ vanish whenever $\psi_m$ and $\psi_n$ have the same parity. This is a simple example of a forbidden transition.
\end{example}
\begin{example}\label{ex:mol2}
(Laser-guided chemical reactions)
Our last model is of central interest in photochemistry, but to our knowledge has not hitherto been considered at all in the mathematical literature.
Consider a molecule with $M$ nuclei at positions $R = (R_1,\dotsc,R_M) \in \R^{3M}$, and $N$ electrons at positions $x_1,\dotsc,x_N$ with spins $s_1,\dotsc,s_N$.
The state of the molecule at time $t$ is described by a wave function $\Psi \in L^2(\R^{3M} \times {(\R^3 \times \Z_2)}^N)$, i.e.\ $\Psi = \Psi(R, x_1, s_1,\dotsc, x_N, s_N)$.
The evolution equation is given by
\[
i \partial_t \Psi(t) = \left( \sum_{\alpha=1}^M -\frac{1}{2 m_\alpha} \Delta_{R_\alpha} + H_{e\ell}^{(R)} + E(t) \cdot D^{(R)} \right) \Psi(t),
\]
with electronic Hamiltonian $H_{e\ell}^{(R)} = -\frac12 \Delta + V^{(R)}(x_1,\dotsc,x_N)$,
potential $V^{(R)}(x_1,\dotsc,x_N) = - \sum_{i=1}^N \sum_{\alpha=1}^M Z_\alpha / |x_i-R_\alpha| + \sum_{1\le i<j\le N}    1 / |x_i-x_j| + \sum_{1\le\alpha<\beta\le M} Z_\alpha Z_\beta / |R_\alpha-R_\beta|$,
and dipole operator $D^{(R)}(x_1,\dotsc,x_N) = \sum_{i=1}^N e x_i - \sum_{\alpha=1}^M e Z_\alpha R_\alpha$.
Here, $e$ is the electronic charge ($-1$ in atomic units), $-eZ_\alpha$ are the nuclear charges ($+Z_\alpha$ in atomic units), and $m_\alpha$ are the nuclear masses.
A careful mathematical account in the absence of control fields and for smooth interaction potentials is given in~\cite{T:2003:SV}.

\begin{figure}
  \centering
  \includegraphics[width=0.5\textwidth]{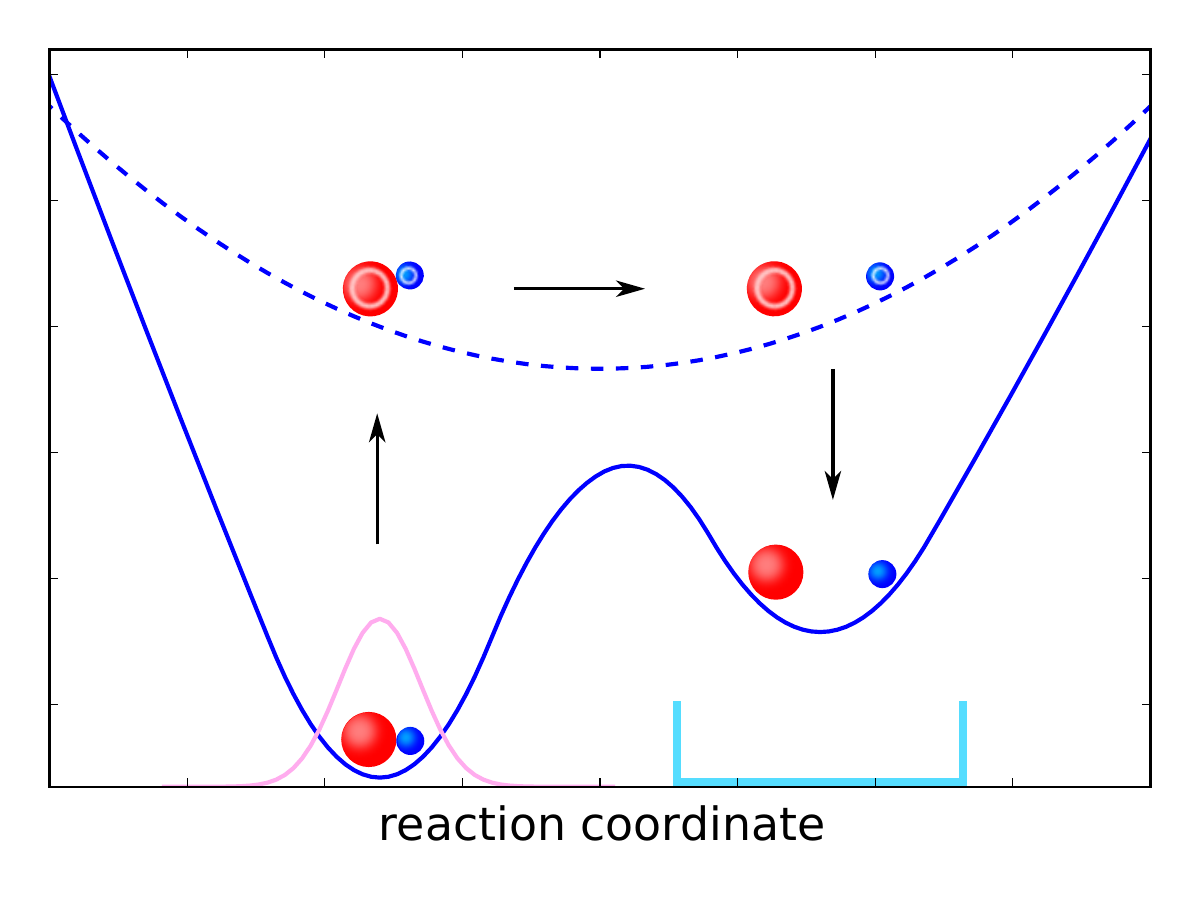}
  \caption{Schematic representation of laser-controlled chemical reaction dynamics. The nuclei of a molecule move on different potential energy surfaces depending on the electronic state, and the laser induces transitions between these states. Blue: Potential energy surfaces. Magenta: Initial wave function of the nuclei. Cyan: Target region.}\label{fig:multiple_pes}
\end{figure}
The typical situation in laser control of chemical reactions is the following, see Figure~\ref{fig:multiple_pes}.
The system starts in a stationary state of electrons and nuclei.
The laser then induces a transition to a different electronic state.
As a result the nuclei now see a different potential energy surface with respect to which they are no longer in equilibrium;
for instance the new surface may no longer contain a barrier to a desired target position.
Once the nuclei have moved barrier-free to the target position, the laser induces a transition back to the original surface so as to also put the electrons in the target state.
Mathematically, this situation can be modelled by generalizing the ansatz for electronic state in Example~\ref{ex:atom} to an ansatz of the joint wave function of electrons and nuclei.
Confining ourselves for simplicity to two electronic states, one assumes
\begin{multline*}
  \Psi(R,x_1,s_1,\dotsc,x_N,s_N, t) \approx \Phi_1(R,t)\psi_1^{(R)}(x_1,s_1,\dotsc,x_N,s_N)\\ + \Phi_2(R,t)\psi_2^{(R)}(x_1,s_1,\dotsc,x_N,s_N),
\end{multline*}
where $\Phi_1$, $\Phi_2\in L^2(\R^{3M})$ are nuclear wave functions, $||\Phi_1||^2+||\Phi_2||^2=1$, and $\psi_1^{(R)}$, $\psi_2^{(R)}$ are normalized eigenstates of the electronic Hamiltonian $H_{e\ell}^{(R)}$.
This leads to the following Schr\"odinger equation in the Hilbert space $\mc H=L^2(\R^{3M};\C^2)$
\begin{multline}\label{eq:twolevel}
    i\partial_t
    \begin{pmatrix}\Phi_1 \\ \Phi_2 \end{pmatrix}(t) =
    \left( \sum_{\alpha=1}^M -\frac{1}{2m_\alpha}\Delta + \begin{pmatrix} E_1(R) & 0 \\ 0 & E_2(R) \end{pmatrix}\right.\\ +
         \left. \vphantom{\sum_{\alpha=1}^M} \begin{pmatrix} E(t)\cdot \mu_{11}(R) & E(t)\cdot \mu_{12}(R) \\
         E(t)\cdot{\mu_{12}(R)}^* & E(t)\cdot\mu_{22}(R) \end{pmatrix}\right)
     \begin{pmatrix}\Phi_1 \\ \Phi_2 \end{pmatrix},
\end{multline}
with the dipole moment functions
\begin{equation}\label{eq:diptwolevel}
    \mu_{ij}(R) = \left\langle \psi_i^{(R)}, - D^{(R)} \psi_j^{(R)}\right\rangle_{\mc H_{e\ell}}.
\end{equation}
\end{example}
\section{Cost functionals and functional analytic setting}\label{sec:control}
To identify control fields $v_\ell \colon [0,T] \to \R$ which achieve a suitable goal, such as transfer of the system from one eigenstate of $H_0$ to another, and which are convenient to implement experimentally, one typically minimizes a cost functional which promotes goal achievement and penalizes unsuitable controls~\cite{HTKSVR:2012:PCCP}.
The theoretical work up to now has focused on cost functionals like the $L^2$ norm of the control field (or variants thereof like the $L^1$ or $H^1$ norm).
This approach has been very successful in achieving a wide spectrum of control goals.
But it has been remarked that ``[the] resulting optimized electric fields are often very complex, thus it is nearly impossible to understand the underlying processes involved in the observed control''~\cite{HTKSVR:2012:PCCP}.
We may therefore ask the question: \emph{Why $L^2$} (or variants thereof)?

We will argue that using certain more sophisticated costs leads to ``simpler'' controls. The control fields are \emph{sparse in frequency}, picking out the system's Bohr frequencies without any special ansatz, while at the same time having \emph{slowly varying amplitude envelopes}, thereby sharing an important feature of laser pulses designed by experimentalists which are commonly and successfully used in the laboratory.

We first state the general class of costs we propose, including its functional analytic setting.
Subsequently we give examples.
The simplest example of such a cost which yields frequency-sparse controls, and which motivated the general setting, is described in Example~\ref{ex:bloch_h1} below.

\subsection{General setting}
For the quantum system
\begin{equation}\label{eq:state}
i\partial_t \psi(t) = \big(H_0 + \sum_{l=1}^L {Bu(t)}_l H_l\big)\psi(t), \quad \psi(0) = \psi_0
\end{equation}
we consider the optimal control problem
\begin{equation}\label{eq:opt_prob}
\min_{\psi,\, u} J(\psi, u) \quad \text{subject to \eqref{eq:state}}
\end{equation}
where
\begin{equation}\label{eq:cost}
J(\psi, u) = \frac{1}{2} \la \psi(T), \mc O \psi(T) \ra_\mc H + \alpha \lVert u \rVert_{\mc M(\Omega;\,\mc U)}.
\end{equation}
The functional $J$ consists of the term $\frac{1}{2} \la \psi(T), \mc O \psi(T) \ra_\mc H$ which describes the expectation value that needs to be minimized, and a cost term of the form $\alpha \lVert u \rVert_{\mc M(\Omega;\,\mc U)}$ which contains a measure norm explained below and forces the solution to be sparse in a suitable sense.

We make the following very general functional-analytic assumptions on the operators and fields appearing in~\eqref{eq:state}--\eqref{eq:opt_prob}.
These assumptions cover all the physical examples from Section~\ref{sec:models}.

\textbf{ 1. Dynamics.}
Assume that the Hamiltonian $H_0$ is a self-adjoint operator on a Hilbert space $\mc H$ with domain $\mc D(H_0)$.
The initial condition $\psi_0$ may be any element of $\mc H$.
The coupling operators $H_\ell$ are assumed to be bounded self-adjoint operators on $\mc H$.
We use the vector operator notation
$v \cdot \tilde H = \sum_{l=1}^L v_l (-iH_l)$,
$v \cdot \tilde H^* = \sum_{l=1}^L v_l {(-iH_l)}^*$,
$\la \chi_1, \tilde H \chi_2 \ra_\mc H = {(\la \chi_1, -iH_l \chi_2 \ra_\mc H)}_{l=1}^L \in \R^L$
for $\tilde H = (-iH_1,\dotsc,-iH_L)$, $v \in \R^L$ and $\chi_1, \chi_2 \in \mc H$.

For the admissible class of controls $u$ and control operators $B_\ell$ specified below, we will show that eq.~\eqref{eq:state} possesses a unique mild solution in the state space of continuous paths in the Hilbert space, $C([0, T]; \mc H)$.

\textbf{ 2. Target constraint.}
The observable $\mc O$ specifying the target constraint can be any bounded self-adjoint operator on $\mc H$.
Typically $\mc O$ is the orthogonal projection onto the subspace we want to reach.
If $\mc O$ is a projection, the target constraint contribution $\frac{1}{2} \la \psi(T), \mc O \psi(T) \ra_\mc H$ to the cost lies in the interval
$[0, 1/2]$.
The value $0$ corresponds to a $100\%$ achievement of the control objective, and the value $0.5$ to a $0\%$ achievement.

\textbf{ 3. Control space, cost, measure norm.}
The control field $u$ is assumed to belong to a measure space of form
\begin{equation}\label{measurespace}
  \mc M(\Omega; \mc U),
\end{equation}
where $\Omega$ is a locally compact space (typically a closed interval of admissible frequencies), and $\mc U$ is a separable Hilbert space of time-dependent functions (admitting general separable Hilbert spaces can be useful to obtain nice optimality conditions, see the discussion after Example~\ref{ex:gabor_l2}).
The space~\eqref{measurespace} is the space of Borel measures $u$ on $\Omega$ with values in $\mc U$ of finite mass norm $||u||_{\mc M}$.
The mass norm of the measure $u$ is the second term in the cost functional $J$ in~\eqref{eq:cost}.

The space~\eqref{measurespace} is the dual of the separable space $C_0(\Omega; \mc U)$ of continuous functions on $\Omega$ with values in $\mc U$ which can be uniformly approximated by functions with compact support.
The duality pairing is given by
\begin{equation}
\la u, \varphi \ra_{\mc M, \mc C_0} = \int_\Omega \la u'(\omega), \varphi(\omega) \ra_\mc U \diff{\lvert u \rvert(\omega)}
\end{equation}
(see~\cite{M:2009:AMUC}) where $u'$ is the Radon--Nikodym derivative of $u$ with respect to the total variation measure $\lvert u \rvert$ (see~\cite[VII Thm.~4.1]{L:1993:SV}).
Note that the inner product in the integral is the Hilbert space inner product in $\mc U$.
This duality will be very useful.

\textbf{ 4. Control operator.} The control operator is assumed to be a bounded linear operator
\[
B \colon \mc M(\Omega; \mc U) \to L^p(0, T; \R^L) \quad \text{for some $1 < p \leq \infty$}.
\]
Moreover we assume that $B$ has a bounded linear predual operator $B^*\colon L^q(0, T; \R^L) \to C_0(\Omega; \mc U)$, by which we mean a bounded linear operator such that $\la B^* f, u \ra_{C_0, \mc M} = \la f, Bu \ra_{L^q, L^p}$ for all $f \in L^q(0, T; \R^L)$ and all $u \in \mc M(\Omega; \mc U)$ where $\frac{1}{p} + \frac{1}{q} = 1$.

Existence of a bounded linear predual operators implies the weak-$*$--weak{(-$*$)} continuity of $B$.
That is, weak-$*$ convergence of $u_n$ to $u$ in $\mc M(\Omega; \mc U)$ implies weak convergence of $Bu_n$ to $Bu$ in $L^p(0, T; \R^L)$ if $1 < p < \infty$, and weak-$*$ convergence if $p = \infty$.
The case $p=1$ has to be excluded in our framework since we will make use of the reflexivity of $L^p(0, T; \R^L)$.

Note that the operator $B^*$ depends on the Hilbert space structure of $\mc U$, see examples below.
Since $B^*$ appears in the optimality system, the freedom to choose $\mc U$ can be used to generate nice optimal controls.

All spaces --- possibly containing complex valued objects --- are equipped with a real Banach or Hilbert space structure.
For complex spaces this amounts to always using the real part of the scalar product, i.e.\ we take $\la \varphi, \psi \ra_\mc H$ to mean $\operatorname{Re}\la \varphi, \psi \ra_\mc H$.
In particular, linear always means $\R$-linear and the scalar product is real-valued and $\R$-bilinear.

%
%
%
\subsection{Examples}
We now list some examples for choosing the frequency domain $\Omega$, the Hilbert space $\mc U$ of time-dependent functions, and the control operator $B$.
The first example is the prototype for generating sparse controls.
It motivated the general functional analytic setting proposed above, and naturally incorporates physically relevant controls containing finitely many pulses of particular frequencies~\cite{BTS:1998:RMP,ABCLDKA:2002:M3AS,TLR:2004:PRE,SSBK:2010:JCP,RSDTB:2011:PCCP}.
\setcounter{example}{0}
\begin{example}\label{ex:bloch_h1}
(Two-scale synthesis, smooth functions of time)
Here $u$ will be a two-scale time-frequency representation of the laser field amplitude and $B$ will generate the corresponding field.
The control $u$ can a priori contain a continuum of active frequencies, each with its own smooth envelope.
This can be modeled mathematically as follows.
Let $\Omega$ be a closed subset of $\R^+$, $\mc U = H^1_0(0, T; \C)$, and $p = \infty$.
For $u \in L^1(\Omega; H^1_0(0,T;\C))$ we define $B$ to be the two-scale synthesis operator
\begin{equation}\label{eq:B_twoscale_nomeasure}
(Bu)(t) = \operatorname{Re}\int_\Omega u(\omega, t) e^{i\omega t} \diff\omega.
\end{equation}
By approximation, the expression can be extended to measure-valued controls.
This extension is important in practice, because it allows sharp concentration on a small number of frequencies, and has the following mathematically rigorous integral representation:
\begin{equation}\label{eq:B_twoscale}
  (Bu)(t) = \operatorname{Re}\int_\Omega u'(\omega, t) e^{i\omega t} \diff{\lvert u \rvert(\omega)},
\end{equation}
where $u'$ is the Radon--Nikodym derivative of $u$ with respect to $\lvert u \rvert$.

We remark that this setting naturally contains the physically motivated finite-dimensional ansatz spaces of~\cite{ABCLDKA:2002:M3AS,TLR:2004:PRE,SSBK:2010:JCP,KHK:2010:JOTA} in which a finite number of distinct frequencies can be switched on or off by few-parameter modulation functions:
the field $v(t) = \sum_{j=1}^n v_j(t) \cos(\omega_j t + \phi_j)$
corresponds to $Bu$ with $B$ as in~\eqref{eq:B_twoscale} and $u(\omega, t) = \sum_{j=1}^n \delta(\omega - \omega_j) e^{i\phi_j}v_j(t)$. We allow $\mc U$ to contain complex-valued functions.
This allows phase shifts in the different frequencies without leaving the linear setting.

The predual operator $B^*$ is the solution operator of the second-order boundary value problem
\begin{equation} \label{ODE}
  \frac{\partial^2}{\partial t^2}u(\omega, t) = f(t) e^{-i\omega t}, \quad u\Big|_{t=0} = u\Big|_{t=T}=0,
\end{equation}
i.e.~$B^*f = u$.
The equations~\eqref{ODE} are not coupled for different $\omega$.
The operator $B^*$ is continuous and has the additional regularity $B^*f \in C_0(\Omega; H^2 \cap H^1_0)$.
Here it is important that $\Omega$ is closed.
Otherwise $B$ might not be weak-$*$--weak{(-$*$)} continuous.
\end{example}
\begin{example}\label{ex:gabor_l2}
(Gabor synthesis, $L^2$ functions of time)
In this example we design the control operator $B$ and the Hilbert space $\mc U$ so that the predual operator $B^*$ has a nice form.
We take $\Omega \subset \R^+$, $\mc U = L^2(0, T; \C)$, and define $B$ by
\begin{equation}\label{eq:gabor_operator}
(Bu)(t) = \operatorname{Re} \int_\Omega \int_0^T k(s, t) u'(\omega, s) \diff s\; e^{i\omega t} \diff{\lvert u \rvert(\omega)},
\end{equation}
where $k\colon {[0, T]}^2 \to \R$ is a smooth symmetric kernel.
Roughly, this operator corresponds to a \emph{pre-processing} of the envelopes, with only smoothed envelopes entering the equation.
The predual operator $B^*$ becomes the \emph{Gabor transformation}
\[
(B^*f)(\omega, t) = \int_0^T f(s) k(t, s) e^{-i\omega s} \diff s,
\]
which is a time-frequency representation of the control field.
\end{example}
Example~\ref{ex:bloch_h1} above with $\mc U = H^1_0$ leads to a time-frequency representation $u=B^*f$ which is global in time.
That is, it has a global window equal to the Green's function $G$ of the one-dimensional boundary value problem~\eqref{ODE}, hence $(B^*f)(\omega, t) = \int_0^T f(s) G(t, s) e^{-i \omega s} \diff s$.
On the other hand, Example~\ref{ex:gabor_l2} leads to a time-frequency representation which is local in time, but the definition of $B$ is somewhat complicated.
Note that Example~\ref{ex:gabor_l2} can be reformulated in terms of the control operator from Example~\ref{ex:bloch_h1} by using a different control space.
To see this, let $K\colon L^2(0, T) \to L^2(0, T)$ be the compact operator given by convolution with a Gaussian kernel $k$, $(Kf)(t) = \int_0^T k(t, s) f(s) \diff s$.
Then $K$ is injective and self-adjoint and has an unbounded inverse $A\colon \mc D(A) \to L^2$.
Define $\mc U_k := \mc D(A)$ with the induced scalar product $\la u, v\ra_\mc U = \la A u, A v \ra_{L^2}$.
Then $\mc U_k$ is a Hilbert space and the predual operator $B^*$ of $B$ is $(B^*f)(\omega,t) = \Bigl( K^2(f e^{-i\omega \cdot}) \Bigr)(s)$.
This construction also works for window functions other than a Gaussian.
The equivalence to the choice $\mc U = L^2(0, T; \C)$, with $B$ as in~\eqref{eq:gabor_operator} follows since the dual of the map $X\colon C_0(\Omega; L^2(0, T; \C)) \to C_0(\Omega; \mc U_k)$ defined by $(X\varphi)(\omega) = K\varphi(\omega)$ is an isometric isomorphism between the control spaces that preserves the image under the corresponding control operators.

Our next example shows that our setting also covers interesting cases when $\mc U$ does not contain time-dependent functions.
\begin{example}\label{ex:fourier}
(Fourier synthesis, constant functions of time)
Let $\Omega \subset \R^+$, $\mc U = \C$, and let $B$ be the Fourier synthesis operator, that is to say
\[
(Bu)(t) = \operatorname{Re} \int_\Omega u(\omega) e^{i\omega t}d\omega \Bigl(= \operatorname{Re} \int_\Omega u'(\omega) e^{i \omega t} \diff{\lvert u \rvert(\omega)}\Bigr).
\]
The function $u(\omega)$ here can be viewed as a time-frequency representation $u(\omega, t)$ which is independent of time $t$.
The predual operator is, up to a constant factor, the Fourier transform of functions restricted to $[0,T]$,
\[
(B^*f)(\omega) = \hat f(\omega) = \int_0^T f(t) e^{-i\omega t} \diff t.
\]

An alternative approach to achieving sparsity of a time-global frequency decomposition via an $L^2$ cost combined with iterative `frequency sifting' is given in~\cite{RBKMAHR:2006:JCP}.
\end{example}
\begin{example}\label{ex:gabor}
(Time-frequency Gabor synthesis)
Let $\Omega \subset \R^+ \times [0, T]$ be a subset of time-frequency space,
$\mc U = \C$, and
\[
(Bu)(t) = \operatorname{Re} \int_{\Omega} u'(\omega, s) g_{\omega, s}(t) \diff{\lvert u \rvert(\omega, t)}
\]
for the ansatz function $g_{\omega, s}(t) =  e^{-\frac{(t - s)^2}{2 \sigma^2}}e^{i \omega (t - s)}$ for some $\sigma > 0$.
This defines a suitable extension of the control operator from Example~\ref{ex:gabor_l2} to measures in both time and frequency.
With this control operator, each Dirac measure $u=\delta_{\omega,t}$ corresponds to a Gaussian wave packet centered at time $t$ with frequency $\omega$.
The predual of the control operator is given by
\[
(B^*f)(\omega,t) = \int_0^T \overline{g_{\omega,t}(s)} \, f(s) \diff s.
\]
\end{example}
These examples by no means exhaust our framework, but are meant to give an idea of its flexibility.
We also remark that frequency constraints could easily be included by restricting $\Omega$, compare~\cite{LSTT:2009:PR}.
\section{Theory of the optimal control problem}\label{sec:theory}
In this section we will study the optimal control problem \eqref{eq:opt_prob}.
We first show well-posedness of the problem.
In contrast to~\cite{KPV:2014:SJCO} the main difficulty does not lie in the low regularity of the control since we assume sufficient smoothing of the control operator.
It rather lies in the bilinearity of the state equation together with the low regularity of the state.
Subsequently we derive necessary optimality conditions.
We shall show that the support of the optimal measure can be influenced.
Theorem~\ref{thm:supp_characterization} is the natural analog of Theorem~2.12 in~\cite{KPV:2014:SJCO}.
Differences arise due to the bilinearity of the equation and the non-trivial control operator.
We shall also note interesting relationships between the choices for $B$ and $\mc U$.
Throughout this section we will stay in the setting of mild solutions.
This suggests to develop a derivation of the necessary optimality conditions which only requires integral manipulations and no further regularity discussion for the primal and dual state are necessary.

\subsection{Existence and compactness properties for the state equation}
Here we recall well-known properties of the state equation~\eqref{eq:state} needed to prove the existence of solutions to the optimal control problem.
Throughout this subsection, for a given control field $v \in L^1(0, T; \R^L)$ we  consider mild solutions to the state equation~\eqref{eq:state} i.e.\ functions $\psi$ which satisfy $\psi \in C([0, T]; \mc H)$ and
\begin{equation}\label{eq:state_mild}
\psi(t) = G(t) \psi_0 + \int_0^t G(t - s) v(s) \cdot \tilde H \psi(s) \diff s.
\end{equation}
Here $G$ is the unitary group generated by the skew-adjoint operator $-iH_0$ and $\tilde H = (-i H_l)_l$.
Existence of mild solutions as well as their differentiability in the direction of the field follow from~\cite[Thm.\ 2.5]{BMS:1982:SJCO}, and a convenient expression for the derivative is given in~\cite[Chapter 2]{LY:1995:BB}.
The existence of mild solutions and their differentiability in the direction of the field are classical results \cite[Theorem 2.5]{BMS:1982:SJCO}.
\begin{proposition}\label{pro:state_existence}
Let $v \in L^1(0, T; \R^L)$. Then~\eqref{eq:state_mild} possesses a unique solution $\psi \in C([0, T]; \mc H)$. Furthermore, the mapping $F\colon L^1(0, T; \R^L) \to C([0, T]; \mc H)$ defined by $F(v) = \psi$ is differentiable. The derivative is given by $F'(v)(\delta v) = \psi'$ where
\begin{equation}\label{eq:state'_mild}
\psi'(t) = \int_0^t \mc G(t, s) \delta v(s) \cdot \tilde H \psi(s) \diff s
\end{equation}
with the evolution operator $\mc G$ defined through $\mc G(t, s) \psi(s) = \psi(t)$ for $t \geq s$.
\end{proposition}
We remark that the evolution operators $\mc G(t, s)$ are unitary.
This can easily be shown by working in the rotating frame, i.e. by studying the function $t \mapsto G^*(t)\psi(t)$ \cite{RS:1975:AP}, but it will not be needed in the following.

One downside of the mild solution framework is that the typical compactness arguments based on additional spatial regularity \cite{IK:2007:SJCO,HMMS:2013:SJCO} cannot be used.
Fortunately there is a powerful replacement that works in a very general setting.
We will use the following compactness result from \cite[Theorem 3.6]{BMS:1982:SJCO}.
\begin{proposition}\label{pro:compactness}
Let $(v_n)_n$ be a sequence in $L^1(0, T; \R^L)$ such that $v_n \rightharpoonup v$.
Then the corresponding solutions $\psi_n$ of \eqref{eq:state_mild} satisfy $\psi_n \rightarrow \psi$ in $C([0, T]; \mc H)$, where $\psi$ is the mild solution corresponding to $v$.
\end{proposition}
\begin{remark}
  The results of Proposition~\ref{pro:state_existence} and~\ref{pro:compactness} hold for general $C_0$-semi groups. The unitarity of the evolution operator is not used for the theoretical results in this work.
\end{remark}
\subsection{Existence of optimal controls}
We will now prove the existence of solutions to the optimal control problem \eqref{eq:opt_prob}.
\begin{theorem}
There exists a global solution $(\bar\psi, \bar u) \in C([0, T]; \mc H) \times \mc M(\Omega; \mc U)$ of \eqref{eq:opt_prob}.
\end{theorem}
\begin{proof}
The result follows from standard reasoning in the calculus of variations.
Let $(\psi_n, u_n)$ be a minimizing sequence,
\begin{equation}\label{eq:inf_seq}
  \lim_n J(\psi_n, u_n) = \inf J(\psi, u).
\end{equation}
Since $\mc O$ is bounded and $\alpha > 0$, the sequence $(u_n)$ is bounded in $\mc M(\Omega; \mc U)$.
Therefore, and because $\mc M(\Omega; \mc U) = C(\Omega; \mc U)^*$ is the dual of a separable Banach space, the sequence $(u_n)_n$ has a weak-$*$ convergent subsequence still denoted by $(u_n)_n$ with limit $\bar u \in \mc M(\Omega; \mc U)$.
The weak-$*$--weak(-$*$) continuity of $B$ implies that $B u_n$ converges to $B\bar u$ weakly in $L^p(0, T; \R^L)$ for some $p > 1$ and thus also for $p = 1$.
By Proposition~\ref{pro:compactness} the corresponding sequence of states $(\psi_n)_n$ satisfies $\psi_n(T) \rightarrow \bar\psi(T)$.
Thus the first summand of $J$ converges, $\la \psi_n(T), \mc O \psi_n(T) \ra \rightarrow \la \bar\psi(T), \mc O \bar\psi \ra$.
The second summand of $J$ is weak-$*$ lower semi-continuous as it is a norm in a dual space.
Thus we obtain $\lim_n J(\psi_n, u_n) \geq J(\bar\psi, \bar u)$.
Together with \eqref{eq:inf_seq} this implies the claim.
\end{proof}

\subsection{Necessary optimality conditions}
For theoretical and numerical purposes we will use the reduced form of \eqref{eq:opt_prob},
\begin{equation}\label{eq:red_opt_prob}
\min_{u} j(u), \quad j(u) = J(\psi(u), u).
\end{equation}
Here $\psi(u)$ denotes the solution of $\eqref{eq:state}$ for the control $u$.

The reduced cost functional can be split into two parts.
A nonlinear differentiable part
\begin{equation}\label{eq:f'}
f(v) = \la \psi(T), \mc O \psi(T) \ra_\mc H
\end{equation}
with $v = Bu$, and a nondifferentiable convex part
\[
g(u) = \alpha \lVert u \rVert_{\mc M}.
\]

In the next lemma we will see that the derivative of the differentiable part $f$ is given by
\[
f'(v) = \la \varphi, \tilde H \psi\ra_\mc H
\]
where we recall from Section~\ref{sec:control} that $\la \varphi, \psi \ra_\mc H$ denotes the real part of the inner product, and $\varphi$ is the mild solution of the dual equation
\[
\begin{aligned}
i\partial_t \varphi(t) &= (H_0 + v(t) \cdot i \tilde H)\varphi(t), \quad \varphi(T) = \mc O \psi(T).
\end{aligned}
\]
Using the evolution operator this can be rewritten as
\begin{equation}\label{eq:dual_mild}
\varphi(t) = \mc G(T, t)^* \mc O \psi(T).
\end{equation}
Under suitable assumptions, the representation~\eqref{eq:f'} of $f'$ can be derived using a Lagrange functional approach, see \cite{PDR:1988:PR,vWBV:2009:SJSC}. Since in our setting a variational formulation is not readily available, we will give a short proof in the present setting of mild solutions.
\begin{lemma}\label{lem:ev_diff}
Let $v, \delta v \in L^1(0, T; \R^L)$, and let $\psi$, $\psi'$ and $\varphi$ be the corresponding solutions of \eqref{eq:state_mild}, \eqref{eq:state'_mild} and \eqref{eq:dual_mild}, respectively.
Then the mapping $f\colon L^1(0, T; \R^L) \to \R$ defined by
\[
f(v) = \frac{1}{2} \la \psi(T), \mc O \psi(T) \ra_\mc H
\]
is continuously differentiable with derivative
\[
f'(v)(\delta v) = \int_0^T \delta v(t) \cdot \la \varphi(t), \tilde H \psi(t)\ra_\mc H \diff t.
\]
\end{lemma}
\begin{proof}
The continuous differentiability of the state and the product rule give continuous differentiability of $f$ and
\[
f'(v)(\delta v) = \la \mc O \psi(T), \psi'(T) \ra_\mc H.
\]
Using~\eqref{eq:state'_mild} and~\eqref{eq:dual_mild} we obtain
\begin{align*}
\la \mc O \psi(T), \psi'(T) \ra_\mc H &= \la \mc O \psi(T), \int_0^T \mc G(T, t) \delta v(t) \cdot \tilde H \psi(t) \diff t \ra_\mc H \\
&= \int_0^T \delta v(t) \cdot \la \mc G(T, t)^* \mc O \psi(T), \tilde H \psi(t) \ra_\mc H \diff t\\
&= \int_0^T \delta v(t) \cdot \la \varphi(t), \tilde H \psi(t) \ra_\mc H \diff t.
\end{align*}
\end{proof}
We can now derive the following optimality condition. It is partially inspired the optimality condition derived in~\cite[Thm. 2.11]{KPV:2014:SJCO} for a linear parabolic control problem of form $\partial_t \psi - \Delta \psi = u$.
\begin{proposition}[Optimality conditions]\label{pro:opt_cond}
Let $\bar u$ be a local minimizer of problem \eqref{eq:red_opt_prob}, and let $\bar\psi, \bar\varphi \in \mc C([0, T]; \mc H)$ be the corresponding solutions of~\eqref{eq:state_mild} and~\eqref{eq:dual_mild} for the control field $B \bar u$.
Then
\begin{equation}\label{eq:opt_contr}
\alpha \lVert \bar{u} \rVert_{\mc M} = -\la B^* \la \bar\varphi, \tilde H \bar\psi \ra_\mc H, \bar{u} \ra_{C_0, \mc M}.
\end{equation}
and
\begin{equation}\label{eq:opt_est}
\lVert B^* \la \bar\varphi, \tilde H \bar\psi \ra_\mc H \rVert_{C_0} \leq \alpha.
\end{equation}
\end{proposition}
\begin{proof}
Since we can split our functional into a sum of a nonconvex and a nonsmooth part, the result can be deduced from the general differential calculus of Clarke \cite{C:1990:SIAM}.
Because this calculus is somewhat intricate and due to the importance of the optimality system, we prefer to give a more elementary proof.

Let $\bar u$ be a local minimizer of problem \eqref{eq:red_opt_prob} and let $\bar\psi$ and $\bar\varphi$ be the corresponding solutions of \eqref{eq:state_mild} and \eqref{eq:dual_mild}.
We first show the variational inequality
\begin{equation} \label{eq:var_ineq}
g(\bar u) - f'(B\bar u)(Bu - B\bar u) \leq g(u).
\end{equation}
Since $\bar u$ is locally optimal, we have
\[
\frac{1}{h} \big(j(\bar u + h(u - \bar u)) - j(\bar u)\big) \geq 0
\]
for $u \in \mc M(\Omega; \mc U)$ and $h \in (0, 1)$ sufficiently small.
Using the decomposition $j = f \circ B + g$ and convexity of $g$, this implies
\[
\frac{1}{h}\big(f(B\bar u + h(Bu - B\bar u)) - f(B\bar u)\big) + g(u) - g(\bar u) \geq 0.
\]
Since $f$ is differentiable, taking the limit $h \to 0$ yields~\eqref{eq:var_ineq}.

Testing~\eqref{eq:var_ineq} with $u = 0$ and $u = 2\bar u$ gives
\begin{equation} \label{eq:opt_contr_abstract}
g(\bar u) + f'(B\bar u)(B\bar u) = 0.
\end{equation}
Substituting \eqref{eq:opt_contr_abstract} into~\eqref{eq:var_ineq} gives
\begin{equation} \label{eq:opt_ineq_abstract}
-f'(B\bar u)(Bu) \leq g(u)
\end{equation}
for all $u \in \mc M(\Omega; \mc U)$.

Using Lemma~\ref{lem:ev_diff} on the derivative of $f$, equation~\eqref{eq:opt_contr_abstract} gives
\[
g(\bar u) = - \la \la \bar\varphi, \tilde H \bar\psi\ra_\mc H, B \bar u \ra_{L^q, L^p} =  - \la B^* \la \bar\varphi, \tilde H \bar\psi\ra_\mc H, \bar u \ra_{C_0, \mc M}
\]
which proves \eqref{eq:opt_contr}.
From~\eqref{eq:opt_ineq_abstract} we obtain
\[
- \la B^* \la \bar\varphi, \tilde H \bar\psi\ra_\mc H, u \ra_{C_0, \mc M} \leq \alpha \lVert u \rVert_\mc M.
\]
Testing this inequality with $u = -\delta_\omega (B^*\la \bar\varphi, \tilde H \bar\psi\ra_\mc H)(\omega)$ for some $\omega \in \Omega$ yields
\[
\lVert (B^* \la \bar\varphi, \tilde H \bar\psi \ra_\mc H)(\omega) \rVert_\mc U^2 \leq \alpha \lVert (B^* \la \bar\varphi, \tilde H \bar\psi \ra_\mc H)(\omega) \rVert_\mc U
\]
which gives~\eqref{eq:opt_est}.
\end{proof}
\begin{remark}
Proposition~\ref{pro:opt_cond} provides only a necessary condition for local optimality.
Due to the nonlinear structure of the problem \eqref{eq:opt_prob}, we expect that there also exist non-optimal critical points of $j$, as well as local optima that are not global.
\end{remark}
Proposition~\ref{pro:opt_cond} implies the following interesting restrictions on the support and direction of the optimal measure.
\begin{proposition}[Characterization of frequency support]\label{thm:supp_characterization}
Let $\bar u$, $\bar\psi$ and $\bar\varphi$ be as in Proposition~\ref{pro:opt_cond}.
Then we have
\begin{align}
\operatorname{supp}\lvert \bar{u} \rvert &\subset \{\; \omega \in \Omega \mid \lVert (B^* \la \bar\varphi, \tilde H \bar\psi \ra_\mc H)(\omega) \rVert_{\mc U} = \alpha \;\}, \label{eq:opt_supp}\\
-\alpha\bar{u}'(\omega) &= (B^* \la \bar\varphi, \tilde H \bar\psi \ra_\mc H)(\omega), \quad \lvert \bar u \rvert\text{-almost everywhere.}\label{eq:opt_dirctn}
\end{align}
\end{proposition}
\begin{proof}
Writing equation \eqref{eq:opt_contr} as an integral yields
\begin{equation}\label{eq:int_b0}
\int_\Omega \Big(\alpha + \la (B^* \la \bar\varphi, \tilde H \bar\psi \ra_\mc H)(\omega), \bar{u}'(\omega) \ra_{\mc U}\Big) \diff{\lvert \bar u \rvert(\omega)} = 0.
\end{equation}
For the integrand we obtain by the Cauchy--Bunyakovsky--Schwarz (CBS) inequality, using $\lVert \bar{u}'(\omega) \rVert = 1$ and \eqref{eq:opt_est},
\begin{equation}\label{eq:intd_b0}
\alpha + \la (B^* \la \bar\varphi, \tilde H \bar\psi \ra_\mc H)(\omega), \bar{u}'(\omega) \ra_{\mc U} \geq \alpha - \lVert (B^* \la \bar\varphi, \tilde H \bar\psi \ra_\mc H)(\omega) \rVert_\mc U  \geq 0.
\end{equation}
Therefore \eqref{eq:int_b0} yields
\[
\alpha + \la (B^* \la \bar\varphi, \tilde H \bar\psi \ra_\mc H)(\omega), \bar{u}'(\omega) \ra_{\mc U} = 0
\]
for $\lvert \bar{u} \rvert$-almost all $\omega \in \Omega$.
For those $\omega$ the CBS inequality in \eqref{eq:intd_b0} was sharp.
This implies \eqref{eq:opt_dirctn} and
\[
\lVert (B^* \la \bar\varphi, \tilde H \bar\psi \ra_\mc H)(\omega) \rVert_{\mc U} = \alpha
\]
for $\lvert \bar{u} \rvert$-almost all $\omega \in \Omega$.
Since $\omega \mapsto \lVert (B^* \la \bar\varphi, \tilde H \bar\psi \ra_\mc H)(\omega) \rVert_{\mc U}$ is continuous this implies \eqref{eq:opt_supp}.
\end{proof}

For specific control operators $B$, equation~\eqref{eq:opt_dirctn} implies additional regularity for $\bar u'(\omega)$.
For example, in the case of the control operator from Example~\ref{ex:bloch_h1} we obtain the regularity $\bar u'(\omega) \in H^2(0, T)$.

The relation \eqref{eq:opt_supp} for the support of the optimal measure gives us the following corollary.
\begin{corollary} \label{cor:compact_supp}
Let $\bar u$ be a local minimizer of \eqref{eq:red_opt_prob}.
Then $\operatorname{supp} \lvert \bar u \rvert$ is compact.
\end{corollary}
\begin{proof}
Since $B^*\la \bar\varphi, \tilde H \bar\psi\ra \in C_0(\Omega; \mc U)$ we know that there is a compact set $K \subset \Omega$ such that $B^*\la \bar\varphi, \tilde H \bar\psi\ra \leq \alpha / 2$ for all $\omega \not\in K$.
Using \eqref{eq:opt_supp} this implies $\operatorname{supp}\lvert \bar u \rvert \subset K$.
Therefore $\operatorname{supp} \lvert \bar u \rvert$ is compact as a closed subset of the compact set $K$.
\end{proof}
This corollary is of significant physical interest.
It says that although the frequency domain $\Omega$ might be unbounded, optimal solutions will always have bounded support.
Controls from experiments, of course, always have this property, because arbitrarily fast oscillations are not realizable.
But it is remarkable that such oscillations in our theoretical controls can be rules out rigorously.

For most of the specific control operators considered in this work, a significantly stronger result holds.
\begin{theorem}[Frequency sparsity]\label{thm:finite_supp}
Let $\bar u$ be a local minimizer of \eqref{eq:red_opt_prob} and let $\Omega$, $\mc U$ and $B$ be given as in one of the Examples~\ref{ex:bloch_h1}, \ref{ex:gabor_l2}, \ref{ex:fourier}.
Then $\operatorname{supp} \lvert \bar u \rvert$ is finite.
\end{theorem}
\begin{proof}
The results follows from \eqref{eq:opt_supp} by an analyticity argument.
In the setting of the Examples~\ref{ex:bloch_h1}--\ref{ex:fourier}, for every $f \in L^q(0, T)$, the map $\omega \mapsto (B^*f)(\omega)$ is real analytic.
This is standard for the Fourier and Gabor transformation, and follows from analytic dependence on $\omega$ of the right hand side in~\eqref{ODE} for the dual two-scale operator.
If $\Omega \neq \R$ then $B^*f$ can be extended in the natural way to an analytic function on $\R$.
In all cases, $(B^*f)(\omega) \to 0$ if $\omega \to \pm \infty$.

For $f = \la \bar \varphi, \tilde H \bar \psi \ra_\mc H$ we obtain the analyticity of $\omega \mapsto \lVert (B^* \la \bar\varphi, \tilde H \bar\psi \ra_\mc H)(\omega) \rVert_{\mc U}^2$, which implies that the value $\alpha^2$ is attained either everywhere in $\R$ or only at discrete points.
Since the function vanishes at infinity and the support is compact by Corollary~\ref{cor:compact_supp}, the support condition~\eqref{eq:opt_supp} implies the finiteness of $\operatorname{supp} \lvert \bar u \rvert$.
\end{proof}
This result implies that the optimal control is in fact a finite sum of Dirac measures in frequency with values in $\mc U$.
Example~\ref{ex:gabor} cannot be treated with the same technique. Although the map $\omega \mapsto (B^*f)(\omega)$ is complex analytic for $f \in L^q(0, T)$ if $\Omega$ is identified with a subset of $\C$, the analyticity does not carry over to to the absolute values.

%
%
%
\section{Numerical results} \label{sec:numerics}
In this section we apply the framework for sparse time-frequency control to different quantum systems.
First we will describe our numerical approach.
This includes a short discussion of the discretization and the regularization of the optimal control problem.
Then we will present two examples.
The first example is the finite-dimensional system from Example~\ref{ex:atom}.
It serves to illustrate basic effects of sparse control of quantum systems.
The second example is the two-level Schr{\"o}dinger system from Example~\ref{ex:mol2}.
The focus in this more challenging example will be the effect of different control spaces on the resulting optimal controls.

\subsection{Numerical approach}
Our numerical approach relies on the following three steps.
First, we discretize the measure space by a finite sum of Dirac measures with values in a finite-dimensional Hilbert space.
Then, we Huber-regularize the corresponding finite-dimensional nonsmooth problem.
Finally, we solve the resulting smooth optimization problem with a quasi-Newton method.

The first step depends on $\Omega$, of course.
In our examples we always have $\Omega \subset \R^k$ for $k \in \{1, 2\}$ and $\Omega$ is an interval or the product of two intervals.
We fix a uniform (tensor) grid $\Omega^h$ and choose measures supported at those points as our ansatz space.
Those measures can always be written as finite sums of Dirac measures.
In this discrete setting the measure norm reduces to an $\ell^1$ norm for the coefficients from $\mc U$ multiplying the Dirac measures.
To obtain a discrete problem we also need to discretize the Hilbert space $\mc U$ and the quantum system.
In our examples we have $\mc U = H^1_0(0, T; \C)$ or $\mc U = L^2(0, T; \C)$ where we use piecewise linear finite elements on a uniform grid with the appropriate discrete norms, or $\mc U = \C$ were we do not need to discretize.
The discrete Hilbert space is denoted by $\mc U^h$.
The control operator $B^h$ maps discrete controls to piecewise linear functions.
The discretization of the quantum system depends on the system at hand.
We approximate the time evolution of the discretized quantum system by a generalized Suzuki--Trotter method, see~\cite{HL:2015:Pre}.
Together we obtain a finite-dimensional optimization problem that is non-smooth and non-convex.

To deal with the nondifferentiability of the norm at the origin we Huber-regularize this non-smooth problem.
We replace the norm of $\mc U$ in the $\ell^1(\mc U^h)$ norm by the function $h^\theta\colon \mc U \to \R$ given by
\[
h^\theta(z) = \begin{cases}
\lVert z \rVert_\mc U - \frac{\theta}{2}, & \text{if } \lVert z \rVert_\mc U > \theta, \\
\frac{1}{2\theta} \lVert z \rVert_\mc U^2, & \text{if } \lVert z \rVert_\mc U \leq \theta,
\end{cases}
\]
for some regularization parameter $\theta$.
The function $h$ has the following two important properties:
it is differentiable, and the derivatives of $h^\theta$ and the derivatives of the norm of $\mc U$ have the same behavior outside of a small neighborhood of zero.
The first property makes the cost functional differentiable.
The second property preserves the sparsity of solutions, in the sense that frequency regions with zero control amplitude turn into regions with control amplitude below the explicit threshold $\theta$.
More precisely we obtain the following theorem which holds e.g.\ in the case $\Omega = [\omega_-, \omega_+] \cap h \Z$ for some mesh size $h > 0$.
\begin{theorem}[Characterization of frequency support after frequency discretization and Huber regularization]\label{thm:supp_characterization_huber}
Let $\Omega$ be a discrete set of admissable frequencies
and choose any Huber regularization parameter $\theta > 0$.
Consider the optimal control problem~\eqref{eq:opt_prob} with the norm $\lVert u \rVert_{\mc M(\Omega; \mc U)}$ replaced by the Huber regularized $\ell^1(\mc U)$ norm $\sum_{\omega \in \Omega} h^\theta(u(\omega))$.
Then optimal controls $\bar u$ satisfy
\begin{equation}\label{eq:huber_supp}
\{\, \omega \in \Omega \mid \lVert \bar u(\omega) \rVert_\mc U \geq \theta \,\} = \{\, \omega \in \Omega \mid \lVert (B^*\la \bar\varphi, \tilde H \bar\psi \ra)(\omega)\rVert_\mc U = \alpha \,\}.
\end{equation}
\end{theorem}
\begin{proof}
In this case the cost is differentiable and the control is a function rather than a measure with respect to frequency, and the first order optimality condition $0 = j'(\bar u)(\delta u)$ acquires the simple form
\[
 0 = (B^*\la \bar\varphi, \tilde H \bar\psi \ra)(\omega) + \alpha \nabla h^\theta(\bar u(\omega))\quad \text{for all $\omega \in \Omega$}.
\]
Since $\nabla h^\theta(z)$ equals $z / \lVert z \rVert_\mc U$ for $\lVert z \rVert > \theta$ and $z / \theta$ for $\lVert z \rVert_\mc U \leq \theta$, the assertion follows.
\end{proof}
In the numerical examples below, only a few frequencies were contained in the superlevel set on the left which numerically replaces the frequency support of $u$.
See e.g.\ Figure~\ref{fig:var_opt_cond_0d}.

We solved the resulting smooth problem with a quasi-Newton method.
Since the dimension of the control space can become quite large with our approach we chose the memory efficient L-BFGS method, see~\cite{NW:2006:SV}.
For the numerical realization gradients for the discretized problems were used.
The optimization method was terminated as soon as the $\ell^2(\mc U^h)$ norm of the gradient relative to the largest gradient was below a tolerance of at least $10^{-6}$.

The cost parameter is always chosen such that one achieves at least $95\%$ of the control objective, i.e.\ $\frac{1}{2} \la \psi(T), \mc O \psi(T) \ra_\mc H \leq 0.025$.
An automated choice of $\alpha$ would be helpful to obtain useful cost parameters for a variety of problems.

The result of this nonlinear optimization problem also depends on the initial guess for the control.
For small $\alpha$, the initial guess for the control uses a fixed element in $\mc U^h$  together with a random complex phase for the different frequencies.
For larger $\alpha$, where such a generic initial guess leads to convergence of the method to suboptimal critical points near the origin, we use optimal solutions for smaller $\alpha$ as initial guess.


%
%
\subsection{Three level system}
As our first example we chose a typical finite-dimensional projection of an atom in
a laser field, see Example~\ref{ex:atom}.
We use the matrices
\[
H_0 =
\begin{pmatrix*}[r]
-2 & 0 & 0\\
0 & -1 & 0\\
0 & 0 & \hphantom{-}2
\end{pmatrix*},
\quad H_1 =
\begin{pmatrix*}
0 & 0 & 1\\
0 & 0 & 1\\
1 & 1 & 0
\end{pmatrix*}
\]
This corresponds for instance to a 1s, 2s, and 3p state.
Note that the coupling matrix elements~\eqref{eq:dip} between the 1s and 2s state vanishes, i.e.\ the transition 1s $\to$ 2s is ``forbidden'' since these states have equal parity.
On the other hand the transitions 1s $\to$ 3p and 3p $\to$ 2s are allowed.
The control objective is to reach the third eigenstate starting from the ground state.
The initial condition and the observation operator are then given by
$\psi_0 = (1, 0, 0)$ and $\mc O = \operatorname{diag}(1, 1, 0)$.
We use a frequency band $\Omega = [2, 5]$ discretized with $100$ grid points.
The expected transition frequencies $\omega_1 = 3$ and $\omega_2 = 4$ are contained in $\Omega$.
We chose a time horizon of $T = 100$ and a time grid with $4096$ points.
The time horizon was chosen large enough to allow for sufficiently many periods with the transition frequencies.
We chose the cost from Example~\ref{ex:bloch_h1}, i.e.\ $\lVert u \rVert_{\mc M(\Omega;\, H^1_0(0, T; \C))}$ with the control operators given by~\eqref{eq:B_twoscale_nomeasure}--\eqref{eq:B_twoscale}, and a cost parameter of $\alpha = 0.1$.
The control objective was achieved to more than $99.999\%$.
\begin{figure}
\centering
\begin{subfigure}[b]{0.30\textwidth}
\includegraphics[width=\textwidth]{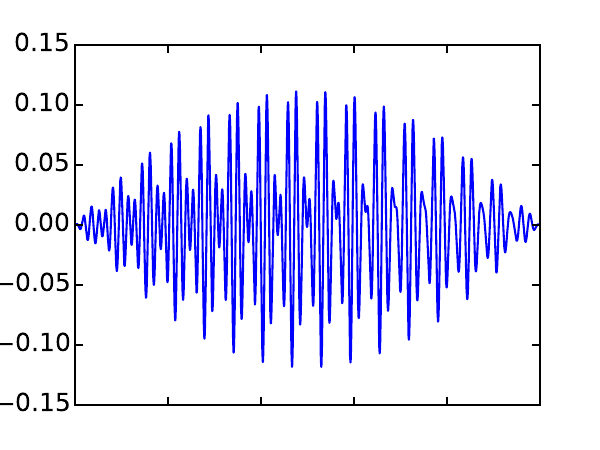}
\caption{}
\label{fig:field_0d_tf}
\end{subfigure}
~
\begin{subfigure}[b]{0.30\textwidth}
\includegraphics[width=\textwidth]{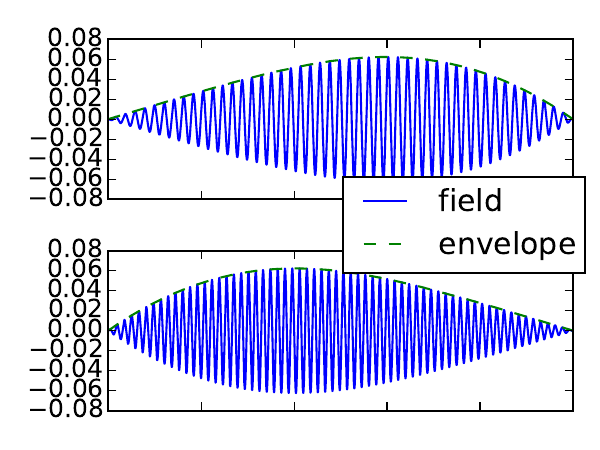}
\caption{}
\label{fig:two_fields_0d_tf}
\end{subfigure}
~
\begin{subfigure}[b]{0.30\textwidth}
\includegraphics[width=\textwidth]{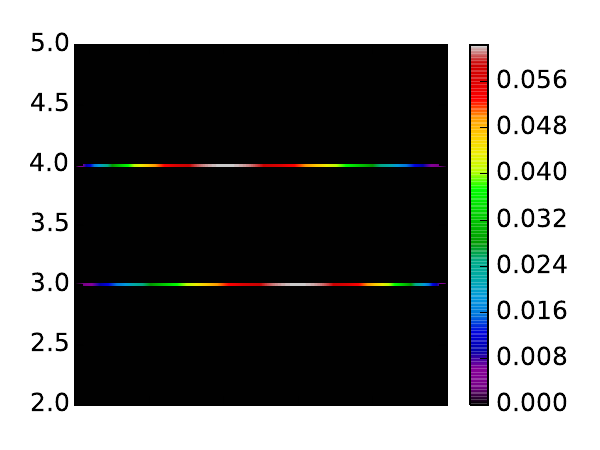}
\caption{}
\label{fig:control_0d_tf}
\end{subfigure}
\caption{%
Optimal control when the cost is chosen as a measure norm with respect to frequency and the $H^1_0$ norm with respect to time (Example~\ref{ex:atom}).
(a) Control field as a function of time.
(c) Time-frequency representation $u(\omega, t)$ (color indicates absolute value).
(b) The contributions due to the two active frequencies of the optimal field.
}
\label{fig:controls_0d}
\end{figure}

The great advantage of the measure space control is that we can decompose the field into simple components.
Figure~\ref{fig:controls_0d} shows the optimal control.
We see that only two frequencies have a visible contribution.
They correspond to the two Bohr frequencies $\omega_1$ and $\omega_2$.
The time profiles are smoothly switched on and off and remain active during the whole time.
This is consistent with the choice $\mc U = H^1_0(0, T; \C)$, which promotes smoothness and non-locality in time.
The field for the first transition reaches its maximum before the field for the second transition.
This reflects the intuitive idea that we have to induce the transition between levels one and three before that between levels three and two.

In fact, unlike most previous control-theory-based forcing fields (for an exception see~\cite{MLT:2007:SJNA} where a weighted $L^2$ regularization is used), the field obtained here bears considerable resemblance to the simple and intuitive few-parameter pulses which have been used by laser physicists for a long time.
Compare, in particular, the two pulses in Figure 9 of~\cite{BTS:1998:RMP}, whose achievement of the control objective was beautifully demonstrated experimentally.

\begin{figure}
\centering
\begin{subfigure}[b]{0.48\textwidth}
\includegraphics[width=\textwidth]{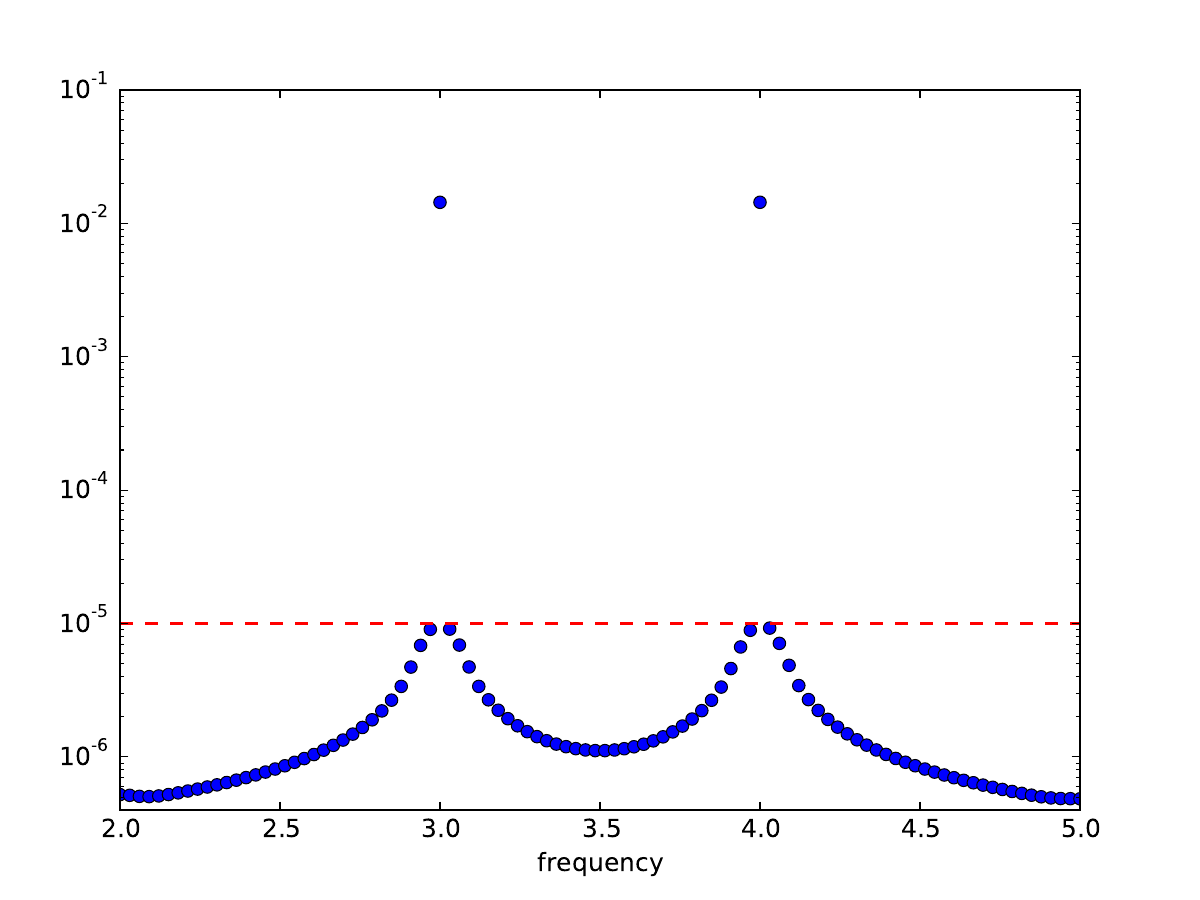}
\caption{}
\label{fig:var_measure}
\end{subfigure}
~
\begin{subfigure}[b]{0.48\textwidth}
\includegraphics[width=\textwidth]{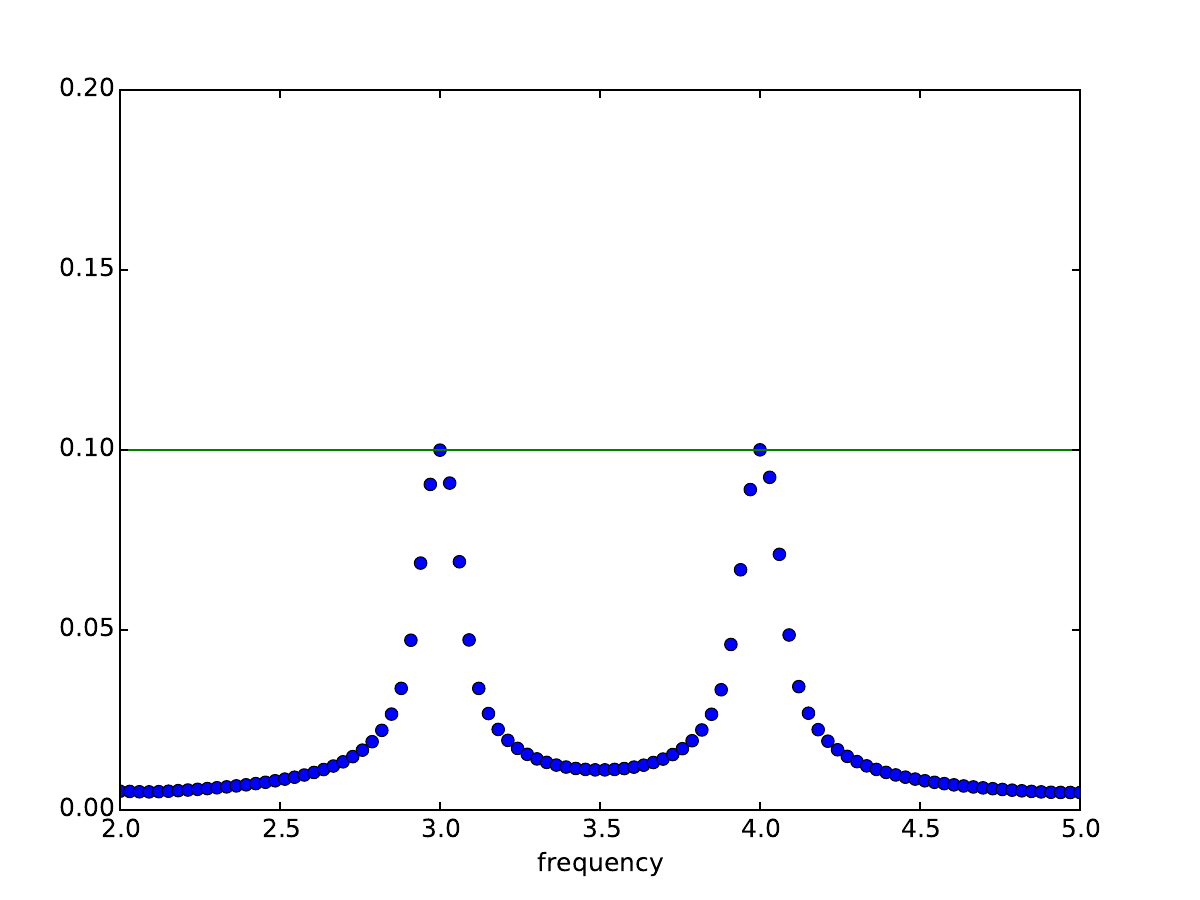}
\caption{}
\label{fig:opt_cond_0d}
\end{subfigure}
\caption{
Detailed numerical illustration of frequency sparsity of the optimal control from Figure~\ref{fig:controls_0d}.
(b) By Theorem~\ref{thm:supp_characterization_huber} the optimal control should vanish off the frequencies where the norm $\lVert (B^* \la \varphi, \tilde H \psi \ra)(\omega) \rVert_{\mc U}$ (dots in (b)) reaches the cost parameter $\alpha$ (solid line in (b)).
(a) The numerical optimal control (dots in (a)), instead of vanishing, drops by three orders of magnitude below the threshold given by the numerical regularization parameter $\theta$ (dashed line in (a)), precisely as theoretically predicted by equation~\eqref{eq:huber_supp}.
}
\label{fig:var_opt_cond_0d}
\end{figure}

Figure~\ref{fig:var_opt_cond_0d} gives a more detailed numerical illustration of the frequency sparsity of the optimal control.
Due to the effect of Huber regularization with $\theta > 0$, the numerical analog of the support of $u$ is the superlevel set $\{\, \omega \mid \lVert u(\omega) \rVert_\mc U > \theta \,\}$ (see equation~\eqref{eq:huber_supp}), which is seen to consist of just two frequencies.
%
%
\subsection{Schr\"odinger dynamics on two potential energy surfaces}\label{sec:numerics_mol2}

In this second example we consider a Schr{\"o}dinger system on two potential energy surfaces as arising in the laser control of chemical reactions, see Example~\ref{ex:mol2}.
The spectral gap between the two potential energy surfaces  varies depending on the position of the nuclear wave function and therefore a much larger variety of frequencies is potentially useful for successful control.
We also expect an additional time structure in the controls due to the movement of the densities in space.
Therefore it is much more challenging to obtain simple controls for this example compared to the previous one.
We focus on a comparison between controls generated for different choices of the space $\mc U$ of time profiles and the control operator $B$.

For simplicity we limit ourselves to one reaction coordinate, i.e.\ space dimension $d=1$.
The control objective is to reach the potential well on the right starting from the potential well on the left, see Figure~\ref{fig:multiple_pes}.
The initial state $\psi_0$ is a Gaussian located in the lower well.
The observation operator $\mc O$ is the projection on the complement of functions with support on the lower surface on the right of the potential barrier.
Instead of the physical coupling given by~\eqref{eq:diptwolevel} we use a coupling operator given by the multiplication with the reaction coordinate on the diagonal and the identity on the off-diagonal.

The energy differences between the two potential energy surfaces measured at the local minima of the lower surface are approximately $0.074$ and $0.048$.
Therefore we expect optimal controls to contain the two frequencies $\omega_1 \approx 0.074$ and $\omega_2 \approx 0.048$.
We chose a time horizon of $T = 3000$ and a time grid with $2048$ points.
The time horizon was chosen large enough to allow for sufficiently many periods with the Bohr frequencies $\omega_1$ and $\omega_2$, and for sufficient movement of the wave function in space.
For the discretization in space we use a simple finite difference scheme on the domain $[-4, 4]$ with $256$ grid points.
We validated that this discretization is sufficient for the range of parameters relevant to this application.

We compare the resulting optimal fields for different choices of the frequency domain $\Omega$, the Hilbert space $\mc U$ of admissible time profiles, and
the control operator $B$.
We also compare them to the optimal
field
for the classical Hilbert space cost functional with $L^2(0, T; \R)$
norm.
In particular we choose the following setups.
\begin{itemize}
\item \textbf{ Two-scale synthesis, smooth functions of time.}
$\Omega = [1/30, 1/10]$, $\mc U = H^1_0(0, T; \C)$, $B$ as in Example~\ref{ex:bloch_h1}.
The frequency band $\Omega$ contains the expected transition frequencies $\omega_1$ and $\omega_2$.
It is discretized with $100$ grid points.
We discretize $\mc U$ with linear finite elements.
The time grid corresponds to the grid of the time stepping.
This results in $100 \cdot 2 \cdot 2048 = 409600$ real degrees of freedom.
\item \textbf{ Gabor synthesis, $L^2$ functions of time.} $\Omega = [1/30, 1/10]$, $\mc U = L^2(0, T; \C)$ and $B$ as in Example~\ref{ex:gabor_l2}, with the Gaussian kernel $k$ suitably adapted to generate homogeneous Dirichlet boundary conditions.
We discretize $\Omega$ and $\mc U$ as before.
For the evaluation of $B$ we explicitly construct the matrix $K$ corresponding to the kernel $k$.
\item \textbf{ Fourier synthesis, constant functions of time.} $\Omega = [1/30, 1/10]$, $\mc U = \C$, $B$ the Fourier synthesis operator, see Example~\ref{ex:fourier}.
The frequency band $\Omega$ is discretized with $100$ grid points.
This results in $2 \cdot 100 = 200$ real degrees of freedom.
\item \textbf{ Time-frequency Gabor synthesis.} $\Omega = [1/30, 1/10] \times [0, T]$, $\mc U = \C$, $B$ as in Example~\ref{ex:gabor}.
The time-frequency cylinder $\Omega$ is discretized by a tensor grid.
In frequency direction we use a regular grid with $100$ grid points.
In time direction we use a grid of $14$ points.
This results in $100 \cdot 14 \cdot 2 = 2800$ real degrees of freedom.
\item \textbf{ Standard $L^2$ cost.} This means we directly minimize over $v \in L^2(0,T;\R)$ the funtional $J(\psi,v)=\frac12\langle\psi(T), \mc O \psi(T))\rangle + \alpha ||v||_{L^2}^2$.
We use linear finite elements on the time grid of the time stepping method.
This results in $2048$ real degrees of freedom.
%
%
\end{itemize}

\begin{figure}
\centering
\hspace{0.7cm}
\parbox{0.30\textwidth}{\centering time}\!
\parbox{0.30\textwidth}{\centering frequency}\!\!\!\!\!\!\!\!\!\!\!
\parbox{0.30\textwidth}{\centering time-frequency\\representation}
\\
\raisebox{.1\height}{\scalebox{0.75}{\rotatebox{90}{\begin{tabular}{c}$\mc M(\Omega; H^1_0)$\\ 2-scale synth.\end{tabular}}}}
\includegraphics[width=0.30\textwidth]{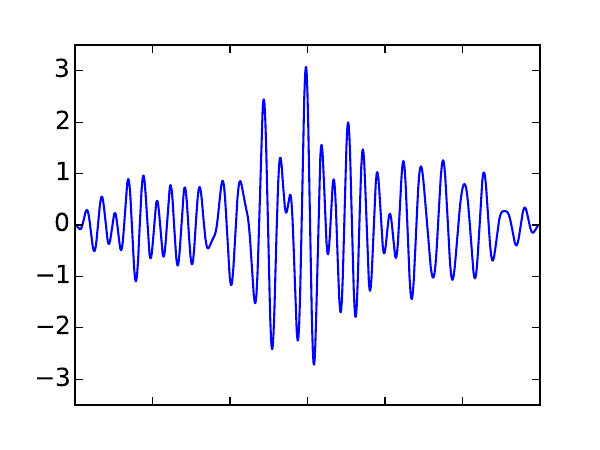}\!
\includegraphics[width=0.30\textwidth]{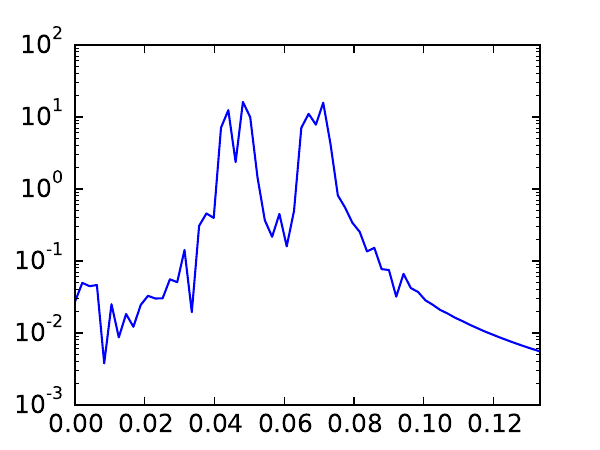}\!\!\!
\includegraphics[width=0.30\textwidth]{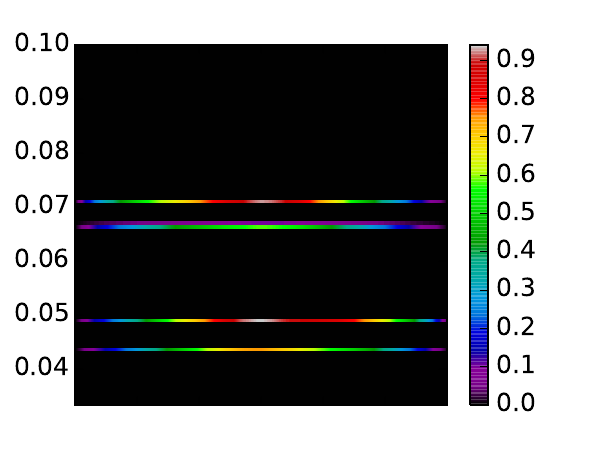}
\\
\raisebox{.2\height}{\scalebox{0.75}{\rotatebox{90}{\begin{tabular}{c}$\mc M(\Omega; L^2)$\\ Gabor synth.\end{tabular}}}}
\includegraphics[width=0.30\textwidth]{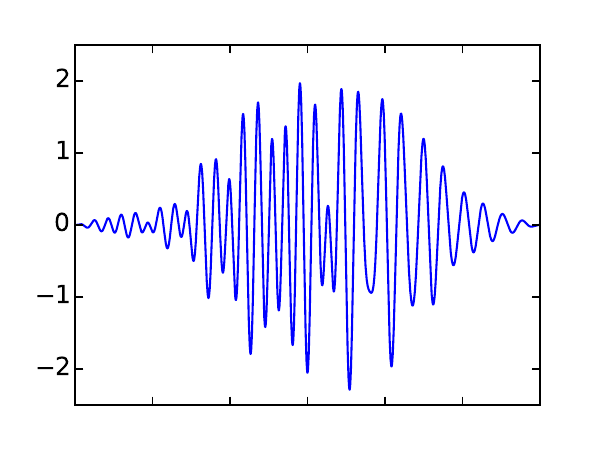}\!
\includegraphics[width=0.30\textwidth]{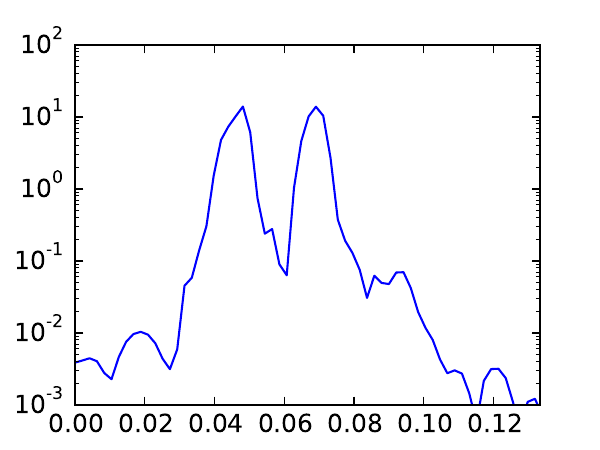}\!\!\!
\includegraphics[width=0.30\textwidth]{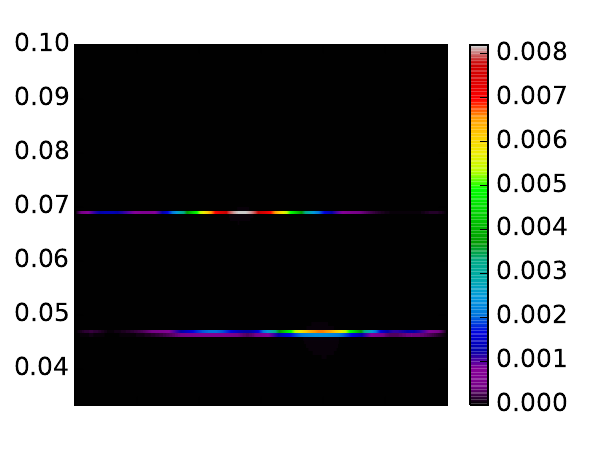}
\\
\raisebox{.1\height}{\scalebox{0.75}{\rotatebox{90}{\begin{tabular}{c}$\mc M(\Omega; \C)$\\ Fourier synth.\end{tabular}}}}
\includegraphics[width=0.30\textwidth]{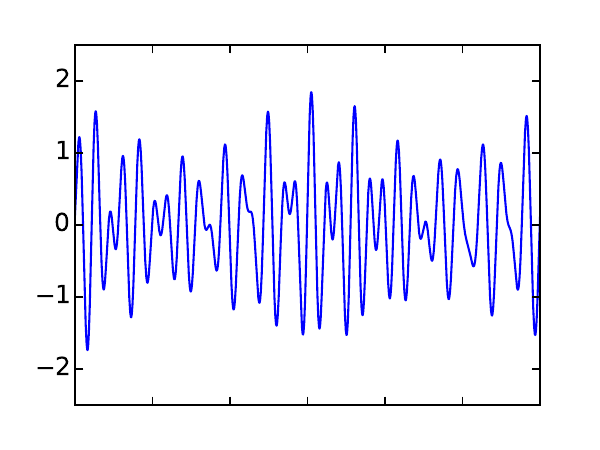}\!
\includegraphics[width=0.30\textwidth]{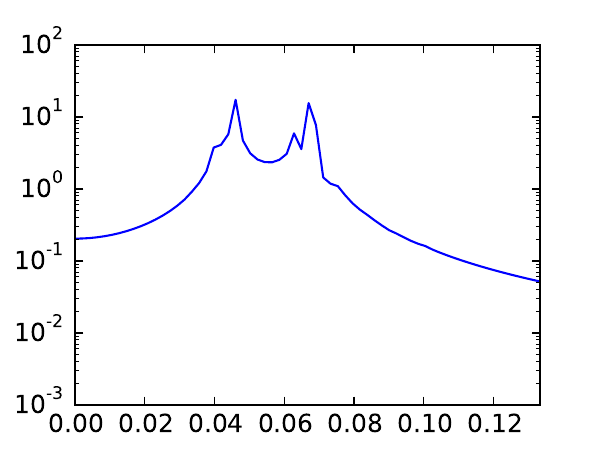}\!\!\!
\includegraphics[width=0.30\textwidth]{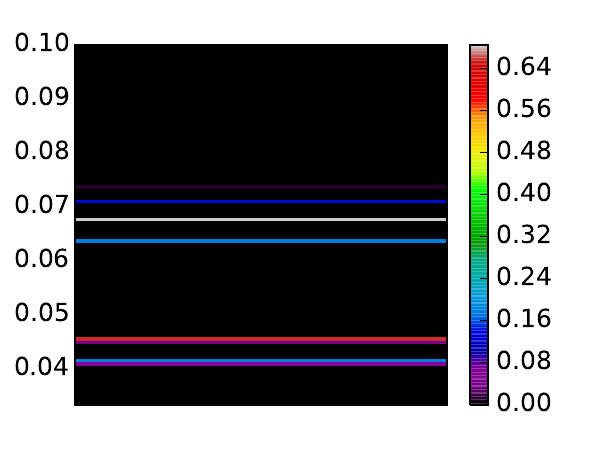}
\\
\scalebox{0.75}{\rotatebox{90}{\begin{tabular}{c}$\mc M(\Omega\times[0, T]; \C)$\\ Gabor synth.\end{tabular}}}
\includegraphics[width=0.30\textwidth]{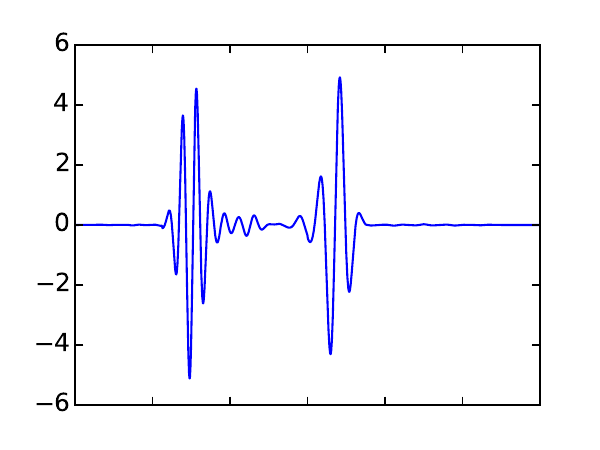}\!
\includegraphics[width=0.30\textwidth]{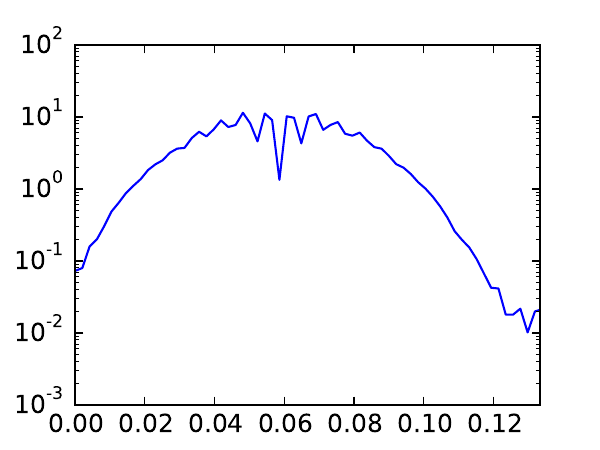}\!\!\!
\includegraphics[width=0.30\textwidth]{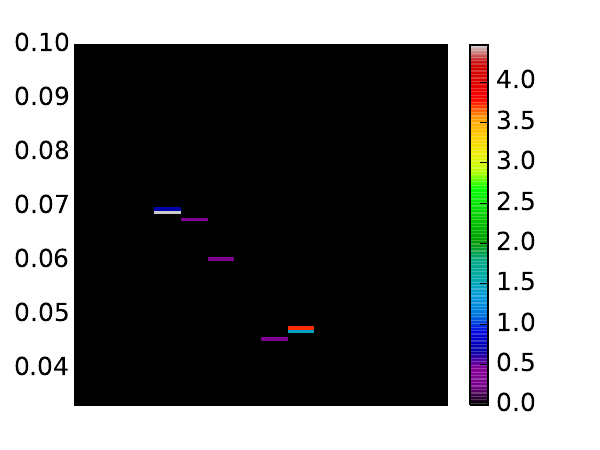}
\\
\raisebox{0.45\height}{\scalebox{0.75}{\rotatebox{90}{\begin{tabular}{c}$L^2(0, T)$\\ $\operatorname{Id}$\end{tabular}}}}
\includegraphics[width=0.30\textwidth]{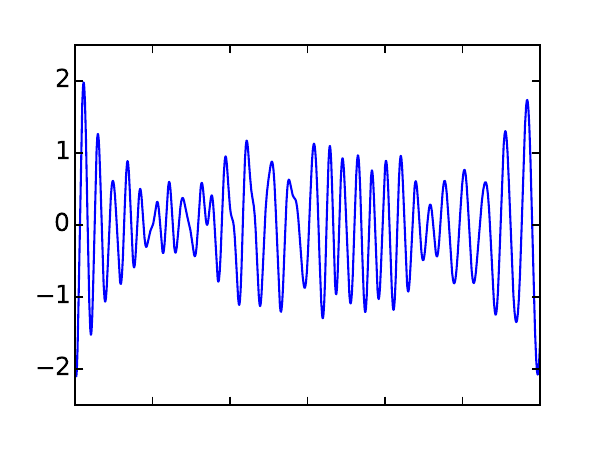}\!
\includegraphics[width=0.30\textwidth]{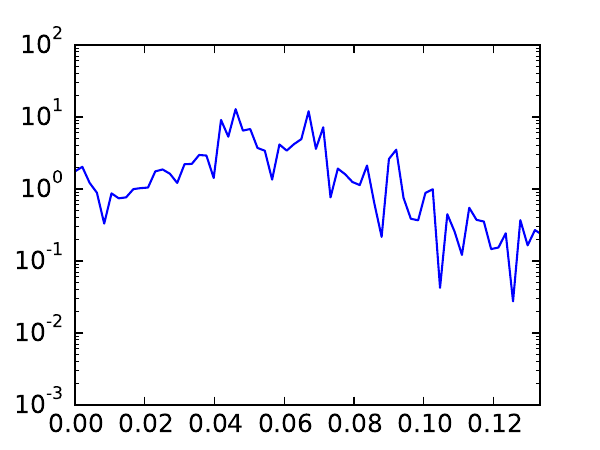}\!\!\!
\includegraphics[width=0.30\textwidth]{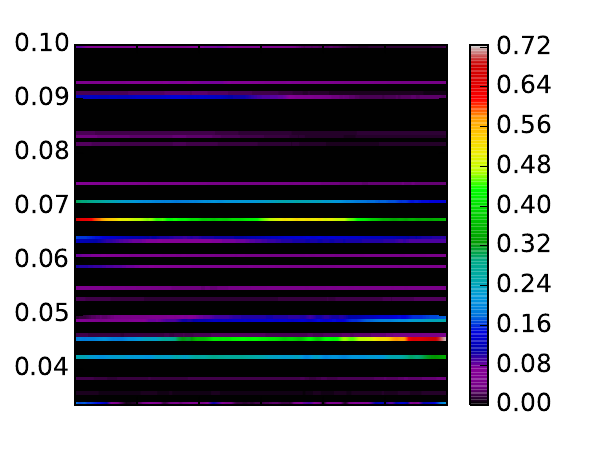}
\caption{
Optimal controls for Schr\"odinger dynamics on two potential energy surfaces.
Rows: Different choices of cost functionals and control operators.
Columns: Time, frequency, and time-frequency representation of the optimal controls (i.e.\ $(Bu)(t)$, $\widehat{Bu}(t)$ and $u(\omega, t)$).
In the rightmost column, the absolute values of the optimal measures are plotted in the time-frequency plane.
}
\label{fig:controls_1d}
\end{figure}

For the simulations we chose $\alpha$ such that probability to end up in the desired subspace is near the value of $95\%$.

Figure~\ref{fig:controls_1d} shows the optimal controls.
We see that their structure depends on the different cost functionals, and that the measure cost~\eqref{eq:cost} for the two-scale and Gabor time-frequency representation leads to an extremely small number of active frequencies.
In the following we will discuss the most significant differences and similarities.

Turning to the second column of Figure~\ref{fig:controls_1d} we focus our attention to the regions around the two Bohr frequencies  $\omega_1$ and $\omega_2$.
The regions correspond to the transitions up from the first well and down into the second well, respectively.
The expected and desired behavior of the frequencies is clearly obtained in the first three choices of the cost functional.
This structure becomes less pronounced for the time-frequency Gabor approach, where the stronger localization of the pulses in time leads to broader regions in the frequency representation, and for the $L^2$ control, which has the frequency profile that is most difficult to interpret.

The third column of Figure~\ref{fig:controls_1d} illustrates the sparse time-frequency structure of the optimal controls.
To compare with the standard $L^2$ control $\bar v$, we also computed a sparse representation $u$ of the latter, by a posteriori minimization of $\lVert Bu - \bar v \rVert_{L^2(0, T)}^2 + \alpha \lVert u \rVert_{\mc M(\Omega; H^1(0, T; \C))}$ with the two-scale synthesis operator $B$ and $\alpha=10^{-4}$.

Figure~\ref{fig:controls_1d} shows that time-frequency representation of the $L^2$ control contains many more active frequencies compared to the other controls.

It is interesting that the space $\mc M(\Omega; H^1_0(0, T; \C))$ leads to more active frequencies compared to the approach with space $\mc M(\Omega; L^2(0, T; \C))$ with the Gabor synthesis operator.
The Fourier and the time-frequency Gabor approaches lead to larger but still reasonably small number of active frequencies.
The choice of a specific cost functional may be guided by the concrete application under consideration.

%
%
%
\section{Conclusion}
In summary, measure valued costs imposed on time-frequency representations of the electric field as introduced in this paper systematically produce frequency sparse controls, in contrast to standard $L^2$ costs.
The resulting controls resemble physically intuitive controls designed by experimentalists.
We hope that the measure-space, time-frequency approach will reduce the current gap between numerical controls on the one side and experimental implementation and physical intuition on the other.
The flexibility of the approach can be further exploited to construct controls suited for concrete experimental setups.
\\[4mm]
\textbf{ Acknowledgements.} The work of F.~Henneke was funded by DFG through the International Research Training Group IGDK 1754.

\bibliographystyle{plain}
\bibliography{sparse_quantum_control}{}

\end{document}